\documentclass[journal]{IEEEtran}

\usepackage{setspace}
\usepackage{comment}

\usepackage{graphicx}
\usepackage{amssymb}
\usepackage{amsmath,amsfonts,amssymb,euscript,epsfig,psfrag,
enumerate,float,afterpage, subfigure}%

\newtheorem{thm}{Theorem}

\newtheorem{lem}{Lemma}

\newcommand{\defequiv}{\mbox{\raisebox{-.3ex}{$\overset{\vartriangle}{=}$}}}

\usepackage{cite}

\begin{document}

\title{Optimal Routing with Mutual Information Accumulation in Wireless Networks}

\author{Rahul Urgaonkar, \emph{Member, IEEE}, and  Michael J. Neely, \emph{Senior Member, IEEE} \\
\vspace{-0.1in}
\thanks{Rahul Urgaonkar is with the Network Research Department, Raytheon BBN Technologies,
Cambridge, MA 02138. Michael J. Neely is with the Department
of Electrical Engineering, University of Southern California, Los Angeles, CA
90089. This work was performed when Rahul Urgaonkar was a PhD student at the
University of Southern California. (Email: rahul@bbn.com, mjneely@usc.edu, 
Web: http://www.ir.bbn.com/$\sim$rahul).}
\thanks{This material is supported in part by one or more of the following:
the DARPA IT-MANET program grant W911NF-07-0028, the NSF Career
grant CCF-0747525, NSF grant 0964479, the Network Science Collaborative
Technology Alliance sponsored by the U.S. Army Research Laboratory
W911NF-09-2-0053.}}

\maketitle
\begin{abstract}
We investigate optimal routing and scheduling
strategies for multi-hop wireless networks with \emph{rateless} codes.
Rateless codes allow each node of the network to accumulate
mutual information from every packet transmission. This enables
a significant performance gain over conventional shortest path
routing. Further, it outperforms cooperative communication
techniques that are based on energy accumulation. However, it
requires complex and combinatorial networking decisions concerning
which nodes participate in transmission, and which decode
ordering to use. We formulate three problems of interest in this
setting: (i) minimum delay routing, (ii) minimum energy routing
subject to delay constraint, and (iii) minimum delay broadcast.
All of these are hard combinatorial optimization problems and
we make use of several structural properties of their optimal
solutions to simplify the problems and derive optimal greedy
algorithms. Although the reduced problems still have exponential
complexity, unlike prior works on such problems, our greedy
algorithms are simple to use and do not require solving any linear
programs. Further, using the insight obtained from the optimal
solution to a line network, we propose two simple heuristics
that can be implemented in polynomial time and in a distributed
fashion and compare them with the optimal solution. Simulations
suggest that both heuristics perform very close to the optimal
solution over random network topologies.

\end{abstract}

\section{Introduction}
\label{section:intro}

Cooperative communication promises significant gains in the performance of wireless networks over
traditional techniques that treat the network as comprised of point-to-point links. 
Cooperative communication protocols exploit the broadcast nature of wireless transmissions and offer 
spatial diversity gains by making use of multiple relays for cooperative transmissions. This can 
increase the reliability and reduce the energy cost of data transmissions in wireless networks.
See \cite{NOW_survey} for a recent comprehensive survey.

Most prior work in the area of cooperative communication has investigated physical layer
techniques such as orthogonal repetition coding/signaling \cite{laneman1}, distributed beamforming \cite{mudumbai}, distributed 
space-time codes \cite{laneman2}, etc. All these techniques perform \emph{energy accumulation} from multiple 
transmissions to decode a packet. In energy accumulation, a receiver can decode a packet when the total
received energy from multiple transmissions of that packet exceeds a certain threshold.
An alternate approach of recent interest is based on \emph{mutual information accumulation} 
\cite{molisch}\cite{icc}. In this approach, a node accumulates
mutual information for a packet from multiple transmissions until it can be decoded
successfully. This is shown to outperform energy accumulation based schemes, particularly
in the high SNR regime, in \cite{molisch}\cite{icc}.
Such a scheme can be implemented in practice using rateless codes of which 
Fountain and Raptor codes \cite{LT, mitzen, Raptor} are two examples. Rateless 
codes encode information bits
into potentially infinite-length codewords. Subsequently, additional parity
bits are sent by the transmitter until the receiver is able
to decode. The decoding procedure at the receiver is carried out using
belief propagation algorithms. For an overview of the working of 
Fountain, LT, and Raptor codes, we
refer to the excellent survey in \cite{mackay}.

In addition to allowing mutual information accumulation, 
rateless codes provide further advantages over traditional  
fixed rate schemes in the context of fading relay networks as discussed in \cite{castura1}\cite{Liu}. 
Unlike fixed rate code schemes in which knowledge of the current channel state information (CSI) is 
required at the transmitters, rateless codes adapt to the channel conditions without requiring CSI. 
This advantage becomes even more important in large networks where the
cost of CSI acquisition grows exponentially with the network size.
However, this introduces  deep
memory in the system because mutual information accumulated from potentially
multiple transmissions in the past can be used to decode a packet.

In this paper, we study three problems on optimal routing and scheduling 
over a multi-hop wireless network using mutual information accumulation. 
Specifically, we first consider a network with a single source-destination pair and 
$n$ relay nodes. When a node transmits, the other nodes accumulate mutual information at a rate
that depends on their incoming link capacity.
All nodes operate under bandwidth and energy constraints as described in detail in
Section \ref{section:model}. We consider three problems in this setting. In the first problem,
the transmit power levels of the nodes are fixed and the objective is to transmit a packet from the source to the destination 
in minimum delay (Section \ref{section:min_delay_routing}). 
In the second problem, the transmit power levels are variable and the objective is to
minimize the sum total energy to deliver a packet to the destination subject to a delay constraint (Section \ref{section:related_problem}). 
In the third problem, we consider the network model with fixed transmit power levels (similar to the first problem) 
and with a single source where the objective is to broadcast a packet to all the other nodes in minimum delay (Section \ref{section:third_problem}).
All of these objectives are important in a variety of networking scenarios.

Related problems of optimal routing in wireless networks with multi-receiver diversity have been studied in 
\cite{teneketzis, laufer09, divbar, dubois} while problems of optimal cooperative diversity routing and broadcasting are
treated in \cite{khandani, sushant, wiopt_10, marjan} and references therein.
Although these formulations incorporate the broadcast nature of wireless transmissions, 
they assume that the outcome of each transmission is a binary success/failure. Further,
any packet that cannot be successfully decoded in one transmission is discarded.
This is significantly different from the scenario considered in this paper where nodes can accumulate
partial information about a packet from different transmissions over time. This can be thought of
as {networking with ``soft'' information}. 

Prior work on accumulating partial information from multiple transmissions
includes the work in \cite{icc, mear, hitch-hiking, maric_broadcast, scaglione, maric_multicast, PAR}.
Specifically, \cite{mear} considers the problem of minimum energy unicast routing in wireless networks
with energy accumulation and shows that it is an NP-complete problem. Similar results are obtained for the
problem of minimum energy \emph{accumulative broadcast} in \cite{hitch-hiking, maric_broadcast, scaglione}.
A related problem of \emph{accumulative multicast} is studied in \cite{maric_multicast}. Minimum energy unicast routing with
energy accumulation only at the destination is considered in \cite{PAR}. 
The work closest to ours is \cite{icc} which treats the minimum delay routing problem
with mutual information accumulation. Both \cite{maric_broadcast}\cite{icc} develop an LP based formulation for their respective
problems that involves
solving a linear program for \emph{every possible ordering} of relay nodes over all subsets of relay nodes to derive the optimal solution.
Thus, for a network with $n$ relay nodes, this exhaustive approach requires solving
$\sum_{m=1}^{n} \binom{n}{m} m! > n!$ linear programs. 

The primary challenge associated with solving the problems addressed in this paper is their inherent combinatorial nature. 
Unlike traditional shortest path routing problems, the cost of 
routing with mutual information accumulation depends not only on the
set of nodes in the routing path, but also their \emph{relative ordering} in the transmission sequence,
making standard shortest path algorithms inapplicable. Therefore, we approach the problem differently.
To derive the optimal transmission strategy for the first problem, we first formulate
an optimization problem in Section \ref{section:formulation} that optimizes over all possible transmission orderings
over all subsets of relay nodes (similar to \cite{maric_broadcast}\cite{icc}). 
This approach clearly has a very high complexity of $O(n!)$.
Then in Section \ref{section:central}, we prove a key structural property of the optimal solution 
that allows us to simplify the problem and derive a simple greedy algorithm that only needs to optimize over all subsets of nodes. 
Further, it does not require solving any linear programs.
Thus, it has a complexity of $O(2^n)$. We derive a greedy algorithm of the same complexity for the second problem in 
Section \ref{section:related_problem}. We note that this complexity, while still exponential, is a significant improvement over solving $O(n!)$ linear programs.
For example, with $n=10$, this requires $2^{10} = 1024$ runs of a simple greedy algorithm 
as compared to  $10! = 3628800$ runs of an LP solver. 
Note that for small networks, (say, $n \leq 10$), it is reasonable to use our algorithm to exactly compute the optimal solution.
Further, for larger $n$ it provides a feasible way to compute the
optimal solution as a benchmark when comparing against simpler heuristics.

For the minimum delay broadcast problem, we identify a similar structural property of the optimal solution in 
Section \ref{section:third_problem} that allows us to simplify the problem and derive a simple greedy algorithm.
While this greedy algorithm still has a complexity of $O(n!)$, it does not require solving any linear programs and thus
improves over the result in \cite{maric_broadcast} that requires solving $n!$ linear programs.
In general, we expect all these problems to be NP-complete based on the results in \cite{mear, hitch-hiking, maric_broadcast, scaglione}.
For the special case of a line topology,  we derive the exact optimal solution in Section \ref{section:line_network}. 
Finally, in Section \ref{section:distributed}, we propose two simple heuristics that can be implemented in
polynomial time and in a distributed fashion and compare them with the optimal solution.
Simulations suggest that both heuristics perform quite close to the optimal solution over
random network topologies.

Before proceeding, we note that the techniques we apply to get these
structural results can also be applied to similar 
problems that use energy accumulation instead of mutual information accumulation.

\begin{figure}
\centering
\includegraphics[height=4cm, width=6.5cm, angle=0]{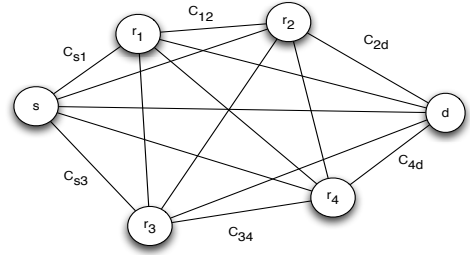}
\caption{Example network with source, destination and $4$
relay nodes. When a node transmits, every other node that has not yet decoded
the packet accumulates mutual information at a rate given by the capacity of the link between the transmitter and that node.}
\label{fig:one}
\end{figure}

\section{Network Model}
\label{section:model}


The network model consists of a source $s$, destination $d$ and $n$ relays $r_1, r_2, \ldots, r_n$ as shown in Fig. \ref{fig:one}. 
There are no time variations in the channel states. 
This models the scenario where the 
coherence time of the channels is larger than any considered
transmission time of the encoded bits.
In the first two problems, the source has a packet to be delivered to the destination. 
In the third problem, the source packet must to delivered to all nodes in the network.

Each node $i$ transmits at a fixed power spectral density (PSD) $P_i$ (in units of joules/sec/Hz) that is uniform across
its transmission band. However, the transmission duration for a node is variable and is a design parameter. The total 
available bandwidth is $W$ Hz. 
A node can transmit the packet only if it has fully decoded the packet. For this, it must accumulate 
\emph{at least} ${I}_{max}$ bits of total mutual information. 

All transmissions happen on orthogonal channels in time or frequency and at most one node can
transmit over a frequency channel at any given time. 
The channel gain between nodes $i$ and $j$ is given by $h_{ij}$ {which is independent of the frequency channel}.
However, this is not necessarily known at the transmitting nodes.

Under this assumption, the minimum transmission time under the two orthogonal schemes 
(where nodes transmit in orthogonal time vs. frequency channels) is the same. 
In the following, we will focus on the case where transmissions are orthogonal in time.
When a node $i$ transmits, every other node $j$ that does not have the full packet yet, receives mutual information at a rate that depends
on the transmission capacity $C_{ij}$ (in units of bits/sec/Hz) of link $i-j$. This transmission capacity itself 
depends on the transmit power $P_i$ and channel strength $h_{ij}$. For example, for an AWGN channel, using Shannon's formula,
this is given by $C_{ij} = \log_2\Big[1 + \frac{h_{ij} P_i}{N_0}\Big]$ where $N_0/2$ is the PSD of the noise process. 
Specifically, if node $i$ transmits for duration $\Delta$ over bandwidth $W$, then node $j$ accumulates $\Delta W C_{ij}$  bits of
information. In the following, we assume $W=1$ for simplicity.
We assume that nodes use independently 
generated \emph{ideal} rateless codes so that the
mutual information collected by a node from different transmissions add up. 
We can incorporate the non-idealities of practical rateless codes by
multiplying $C_{ij}$ with a factor $1/(1 + \epsilon)$ where $\epsilon \geq 0$ is the overhead. A similar model has been considered in \cite{icc}.

\section{Minimum Delay Routing}
\label{section:min_delay_routing}

Under the modeling assumptions discussed in Section \ref{section:model},  the problem of routing a packet from the 
source to the destination with minimum delay consists of the following sub-problems: 
\begin{itemize}
\item First, which subset of relay nodes should take part in forwarding the packet? 
\item Second, in what order should these nodes transmit? 
\item And third, what should be the transmission durations for these nodes? 
\end{itemize}
We next discuss the transmission structure of a general policy under this model.

\subsection{Timeslot and Transmission Structure}
\label{section:timeslot}

\begin{figure}
\centering
\includegraphics[height=3.5cm, width=8.5cm, angle=0]{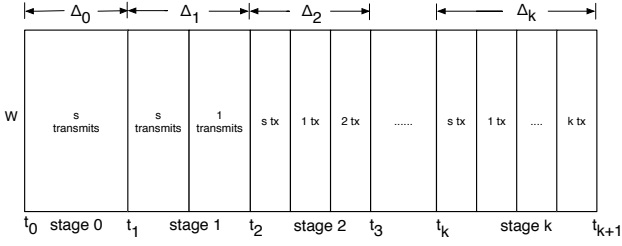}
\caption{Example timeslot and transmission structure. 
In each stage, nodes that have already decoded the full packet
transmit on orthogonal channels in time.}
\label{fig:two}
\vspace{-0.2in}
\end{figure}

Consider any transmission strategy $\mathcal{G}$ for routing the packet to the destination in the model described above. 
This includes the choice of the relay set, the transmission order for this set, and the transmission
durations for each node in this set.
Let $\mathcal{R}$ denote the subset of relay nodes that take part in the routing process
under strategy $\mathcal{G}$. By this, we mean that each node in 
$\mathcal{R}$ is able to decode the packet before the destination and then transmits for a non-zero duration. 
There could be other nodes that are able to decode the packet before the
destination, but these do not take part in the forwarding process and are therefore not included in the set $\mathcal{R}$. 

Let $k = |\mathcal{R}|$ be the size of this set. 
Also, let $\mathcal{O}$ be the ordering of nodes in $\mathcal{R}$
that describes the sequence in which nodes in $\mathcal{R}$ successfully decode the packet under strategy $\mathcal{G}$.
Without loss of generality, let the relay nodes in the ordering $\mathcal{O}$ be indexed
as $1, 2, 3, \ldots, k$. 
Also, let the source $s$ be indexed as $0$ and the destination $d$ be indexed as $k+1$. 
Initially, only the source has the packet.
Let $t_0$ be the time when it starts its transmission and let $t_1, t_2, \ldots, t_k$ denote the times 
when relays $1, 2, \ldots, k$ in the ordering $\mathcal{O}$ accumulate enough mutual information to decode the packet.
Also, let $t_{k+1}$ be the time when the destination decodes the packet.
By definition, $t_0 \leq t_1 \leq t_2 \leq \ldots \leq t_k \leq t_{k+1}$.
We say that the transmission occurs over $k+1$ stages,
where stage $j, j \in \{ 0, 1, 2, \ldots, k\}$ represents the interval $[t_j, t_{j+1}]$.
The state of the network at any time is given by the set of nodes that have the full packet and the
mutual information accumulated so far at all the other nodes.
Note that in any stage $j$, the first $j$ nodes in the ordering 
$\mathcal{O}$ and the source have the fully decoded packet. Thus, any subset of these nodes (including potentially all of them) 
may transmit during this stage. Then the time-slot structure for the transmissions can be depicted as in Fig. \ref{fig:two}.
Note that in each stage, the set of relays that have
successfully decoded the packet increases by one (we ignore those relays that are not part of the set $\mathcal{R}$).

We are now ready to formulate the problem of minimum delay routing with mutual information accumulation.

\subsection{Problem Formulation}
\label{section:formulation}

For each $j$, define the duration of stage $j$ as $\Delta_j = t_{j+1} - t_j$. Also, let $A_{ij}$ denote the transmission 
duration for node $i$ in stage $j$ under strategy $\mathcal{G}$. 
Note that $A_{ij} = 0$ if $i > j$, else $A_{ij} \geq 0$. This is because node $i$ does not have the full packet until the end of stage $i-1$.
The total time to deliver the packet to the destination $T_{tot}$ is given by 
$T_{tot} = t_{k+1} - t_0 = \sum_{j = 0}^k \Delta_j$.
For any transmission strategy $\mathcal{G}$ that uses the subset of relay nodes $\mathcal{R}$ with an ordering $\mathcal{O}$, the minimum delay is 
given by the solution to the following optimization problem:
\begin{align}	
\textrm{Minimize:} \; & T_{tot} = \sum_{j = 0}^k \Delta_j \nonumber \\
\textrm{Subject to:}  \; & \sum_{i=0}^{m-1} \sum_{j=0}^{m-1} A_{ij} C_{im} \geq {I}_{max} \; \forall m \in \{1, 2, \ldots, k+1\} \nonumber \\
& \sum_{i=0}^{j} A_{ij} \leq \Delta_j \; \forall j \in \{0, 1, 2, \ldots, k\} \nonumber \\
& A_{ij} \geq 0 \; \forall i \in \{0, 1, 2, \ldots, k\}, j \in \{0, 1, 2, \ldots, k\} \nonumber \\
& A_{ij} = 0 \; \forall i > j \nonumber\\
& \Delta_{j} \geq 0 \; \forall j \in \{0, 1, 2, \ldots, k\}
\label{eq:min_delay}
\end{align}
Here, the first constraint captures the requirement that node $m$ in the ordering  must accumulate at least ${I}_{max}$ amount of mutual information
by the end of stage $m-1$ using all transmissions in all stages up to stage $m-1$. 
The second constraint means that in every stage $j$, the total transmission time for all nodes that have the
fully decoded packet in that stage cannot exceed the length of that stage.

It can be seen that the above problem is a linear program and thus can be solved efficiently for a given
relay set $\mathcal{R}$ and its ordering $\mathcal{O}$. Indeed, this is the approach
taken in \cite{icc} that proposes solving such a linear program for
\emph{every possible ordering of relays} for each subset of the set of relay nodes.
While such an approach is guaranteed to find the optimal solution, 
it has a huge computational complexity of $O(n!)$ linear programs.
In the next section, we show that the above computation can be 
simplified significantly by making use of a structural property of the optimal solution.

\subsection{Characterizing the Optimal Solution of (\ref{eq:min_delay})}
\label{section:central}

\begin{figure}
\centering
\includegraphics[height=3.5cm, width=8.5cm, angle=0]{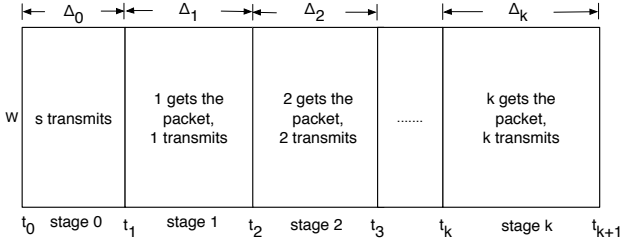}
\caption{Optimal timeslot and transmission structure. 
In each stage, only the node that decodes the packet at the beginning
of that stage transmits.}
\label{fig:three}
\vspace{-0.2in}
\end{figure}

Let $\mathcal{R}_{opt}$ denote the subset of relay nodes that take part in the routing process in the optimal solution. Let 
$k = |\mathcal{R}_{opt}|$ be the size of this set. Also, let $\mathcal{O}_{opt}$ be the optimal ordering.
Note that, by definition, each node in $\mathcal{R}_{opt}$ transmits for a non-zero duration (else, we can remove it from the set without
affecting the minimum total transmission time). 
Then, we have the following:

\begin{thm}
Under the optimal solution to the minimum delay routing problem (\ref{eq:min_delay}), in each stage $j$,
it is optimal for only one node to transmit, and that node is node $j$. 
\label{thm:min_delay}
\end{thm}

Fig. \ref{fig:three} shows the timeslot structure under the optimal solution. 
The above theorem shows that only one node transmits in each stage, and that the optimal transmission
ordering is the same as the ordering that nodes in the set $\mathcal{R}_{opt}$ decode
the packet. Comparing this with the
general timeslot structure in Fig. \ref{fig:two}, it can be seen that 
Theorem \ref{thm:min_delay} simplifies problem ($\ref{eq:min_delay}$) significantly. Specifically,
Theorem \ref{thm:min_delay} implies that, \emph{given the optimal relay set} $\mathcal{R}_{opt}$, the optimal transmission structure 
(i.e., the decoding order and the transmission durations) 
can be computed in a \emph{greedy} fashion as follows. First, the source starts to transmit and continues to do so until any relay node in this set
gets the packet. Once this relay node gets the packet, we know from Theorem \ref{thm:min_delay} that the source does not transmit in 
any of the remaining stages. 
This node then starts to transmit until another node in the set gets the packet. 
This process continues until the destination is able to decode the packet.
The optimal solution to ($\ref{eq:min_delay}$) can then be obtained by applying this greedy
transmission strategy to all subsets of relay nodes and picking one that yields the minimum delay.\footnote{We 
note that the transmission structure characterized by Theorem \ref{thm:min_delay} is similar to the
\emph{wavepath property} shown in \cite{mear} for the problem of minimum energy unicast routing with
energy accumulation in wireless networks. However, our proof technique is significantly different.}  
Note that applying this greedy transmission strategy does not require solving an LP.
While searching over all subsets still has an exponential complexity of $O(2^n)$, 
it can be used to compute the optimal solution as a benchmark. 

Theorem \ref{thm:min_delay} also implies that multiple copies of the packet need not be maintained across the network. For example, note that
the source need not transmit after the first relay has decoded the packet and therefore can drop the packet from its queue.

We emphasize that the optimal transmission structure suggested by Theorem \ref{thm:min_delay} is not obvious.
For example, at the beginning of any stage, the newest addition to the set of relay nodes with the full packet may not
have the best links (in terms of transmission capacity) to all the remaining nodes, including the destination. 
This would suggest that under the optimal solution, in general in each stage, nodes with the full packet should
take turns transmitting the packet. However, Theorem \ref{thm:min_delay} states that such time-sharing is
not required.

Before proceeding, we present a preliminary lemma that is used in the proof of Theorem \ref{thm:min_delay}.
Consider any linear program: 
\begin{align}	
\textrm{Minimize:} \qquad & c^T x \nonumber \\
\textrm{Subject to:}  \qquad & Ax = b \nonumber \\
& x \geq 0
\label{eq:inactive}
\end{align}
where $x \in \mathbf{R}^n$. 
Then we have the following:
\begin{lem}
Let $x^*$ be an optimal solution to the problem (\ref{eq:inactive}) 
such that $x^* > 0$ (where the inequality is taken entry wise). Then 
$x^*$ is still an optimal solution when the constraint $x \geq 0$ is removed.
\label{lem:inactive_const}
\end{lem}

Lemma \ref{lem:inactive_const} implies that removing an \emph{inactive} constraint does not affect the optimal solution of the linear program. 
This is a simple fact and its proof is provided for completeness in Appendix A.


\begin{proof} \emph{(Theorem \ref{thm:min_delay}):} Note that 
Theorem \ref{thm:min_delay} trivially holds in stage $0$ (since only the source has the full packet in this stage).
Next, it is easy to see that in the last stage (i.e., stage $k$), only the node with the \emph{best} link (in terms of transmission
capacity) to the destination in the set $\mathcal{R}_{opt}$ should transmit in order to minimize
the total delay. This is because this node will take the smallest time to transmit the remaining amount of 
mutual information needed by $d$ to decode the packet.
Further, we claim that this node must be the node $k$ in the ordering $\mathcal{O}_{opt}$. This can be argued as follows.
Assume that the node with the best link to the destination in the set $\mathcal{R}_{opt}$ has the full packet 
at some stage $(k-j)$ (where $0 < j < k$) \emph{before} the start of stage $k$. 
Then a smaller delay can be achieved by having only this node transmit after it has decoded the full packet from that
stage onwards. Thus, the other nodes labeled $k-j, \ldots, k-1$ in the transmission order do not transmit, a contradiction.
This shows that under the optimal solution, in the last stage $k$, only node $k$ in the ordering $\mathcal{O}_{opt}$ transmits.
Using induction, we now show that in every prior stage $(k-j)$ where $1 \leq j \leq k-1$, only one node needs to transmit and that
this node must be node $k-j$ in the ordering $\mathcal{O}_{opt}$.

Consider the $(k-1)^{th}$ stage. At time $t_{k-1}$, all nodes except $k$ and $d$ have decoded the packet. Let the mutual information
state at nodes $k$ and $d$ at time $t_{k-1}$ be ${I}_{k}(t_{k-1})$ and ${I}_d(t_{k-1})$ respectively. 
Also, suppose in the $(k-1)^{th}$ stage,
relay nodes $1, 2, \ldots, k-1$ and the source transmit a fraction $\alpha_1^{k-1}, \alpha_2^{k-1}, \ldots, \alpha_{k-1}^{k-1}$ and 
$\alpha_0^{k-1} $ of the total duration of stage $(k-1)$, i.e., $\Delta_{k-1}$ respectively. Note that these
fractions must add to $1$ since it is suboptimal to have any idle time (where no one is transmitting).
Then, the optimal solution must solve the following optimization problem:
\begin{align}
\textrm{Minimize:} \; & \Delta_{k-1} + \Delta_k \nonumber \\
\textrm{Subject to:}  \; & {I}_k(t_{k-1}) + \Delta_{k-1} \sum_{i=0}^{k-1} \alpha_i^{k-1} C_{ik} \geq {I}_{max} \nonumber \\
& {I}_d(t_{k-1}) + \Delta_{k-1} \sum_{i=0}^{k-1} \alpha_i^{k-1} C_{id} + \Delta_k C_{kd} \geq {I}_{max} \nonumber \\
& 0 \leq \alpha_0^{k-1} , \alpha_1^{k-1}, \ldots, \alpha_{k-1}^{k-1} \leq 1 \nonumber \\
& \sum_{i=0}^{k-1} \alpha_i^{k-1} = 1 \nonumber \\
& \Delta_{k-1} \geq 0,  \Delta_k \geq 0 
\label{eq:stagek-1-1}
\end{align}
Here, the first constraint states that relay $k$ must accumulate at least ${I}_{max}$ bits of mutual information
by the end of stage $(k-1)$. 
The second constraint states that the destination must accumulate at least ${I}_{max}$ bits of mutual information
by the end of stage $k$.
Note that in the last term of the left hand side of the second constraint, we have used the fact that only node $k$ transmits during stage $k$.

It is easy to see that under the optimal solution, the first and second constraints must be met with equality.
This simply follows from the definition of the beginning of any stage $j$ as the time when node $j$ has just decoded the packet. 
Next, let $\beta_i = \Delta_{k-1}\alpha_i^{k-1}$ for all $i \in \{0, 1, 2, \ldots, k-1\}$. 
Since $\sum_{i=0}^{k-1} \alpha_i^{k-1} = 1$, we have that $\sum_{i=0}^{k-1} \beta_i = \Delta_{k-1}$ and (\ref{eq:stagek-1-1}) is equivalent to:
\begin{align}
\textrm{Minimize:} \; & \sum_{i=0}^{k-1}\beta_i + \Delta_k \nonumber \\
\textrm{Subject to:}  \; & {I}_k(t_{k-1}) +  \sum_{i=0}^{k-1} \beta_i C_{ik} = {I}_{max} \nonumber \\
& {I}_d(t_{k-1}) + \sum_{i=0}^{k-1} \beta_i C_{id}  + \Delta_k C_{kd} = {I}_{max} \nonumber \\
& \Delta_k \geq 0, \beta_i \geq 0 \qquad \forall i \in \{0, 1, 2, \ldots, k-1\}
\label{eq:stagek-1-3}
\end{align}
Note that problems (\ref{eq:stagek-1-1}) and (\ref{eq:stagek-1-3})
are equivalent because we can transform (\ref{eq:stagek-1-3}) to the original problem by using the relations
$\Delta_{k-1} = \sum_{i=0}^{k-1} \beta_i$ and $\alpha_i^{k-1} = \frac{\beta_i}{\Delta_{k-1}}$. The 
degenerate case where $\Delta_{k-1} = 0$ does not arise because if $\Delta_{k-1} = 0$, then no node transmits in stage $(k-1)$ 
and we transition to stage $k$ in which
only node $k$ transmits. This means node $k-1$ never transmits, 
contradicting the fact that it is part of the optimal 
transmission schedule. 

Since we know that under the optimal solution, $\Delta_k > 0$, we can remove the constraint
$\Delta_k \geq 0$ from (\ref{eq:stagek-1-3}) without affecting the optimal solution (using Lemma \ref{lem:inactive_const}).
Next we multiply the minimization objective in (\ref{eq:stagek-1-3}) by $C_{kd}$ without changing the problem.
Then, using the second equality constraint to eliminate $\Delta_k$ from the objective and ignoring the constant terms, 
(\ref{eq:stagek-1-3}) can be expressed as:
\begin{align}
\textrm{Minimize:} \qquad & \sum_{i=0}^{k-1}\beta_i(C_{kd} - C_{id}) \nonumber \\
\textrm{Subject to:}  \qquad & {I}_k(t_{k-1}) +  \sum_{i=0}^{k-1} \beta_i C_{ik}  = {I}_{max} \nonumber \\
& \beta_i \geq 0 \qquad \forall i \in \{0, 1, 2, \ldots, k-1\} 
\label{eq:stagek-1-4}
\end{align}
This optimization problem is linear in $\beta_i$ with a single linear equality constraint and thus 
the solution is of the form where all except one $\beta_i$ are zero.
Since $\alpha_i^{k-1} = \frac{\beta_i}{\Delta_{k-1}}$, we have that in the optimal solution, exactly one of the fractions 
$\alpha_0^{k-1} , \alpha_1^{k-1}, \ldots, \alpha_{k-1}^{k-1}$
is equal to $1$ and rest must be $0$. This implies that only one node transmits in this stage. Further, 
this node must be the relay node $k-1$ that decoded the packet at the beginning of this stage. Else,
node $k-1$ never transmits. This is because by definition of stage $(k-1)$, node $k-1$ does not have the
packet before the
beginning of stage $(k-1)$ and hence cannot transmit before stage $(k-1)$.
Since only node $k$ transmits when stage $(k-1)$
ends, if node $k-1$ is not the node chosen for stage $(k-1)$, it never
transmits, contradicting the fact that it is part of the optimal set.
\footnote{This is a crucial property 
that holds only for the unicast routing case. As we will see in Section \ref{section:third_problem}, this does not
necessarily hold for the minimum delay broadcast problem.}

\begin{center}
\begin{table*}
{\small
\begin{displaymath}
\left[ \begin{array}{ccc}
\sum_{i=0}^{k-j}\beta_i C_{i,k-j+2} \\
\sum_{i=0}^{k-j}\beta_i C_{i,k-j+3} \\
\vdots \\
\sum_{i=0}^{k-j}\beta_i C_{i,d}
\end{array} \right]
+
\left[ \begin{array}{ccc}
C_{k-j+1, k-j+2} & \ldots & 0\\
C_{k-j+1, k-j+3} & \ldots & 0\\
\vdots 		 & \ddots & \vdots\\
C_{k-j+1, d} 	 & \ldots & C_{k, d}\\
\end{array} \right]
\left[ \begin{array}{cccc}
\Delta_{k-j+1} \\
\Delta_{k-j+2} \\
\vdots \\
\Delta_{k}
\end{array} \right]
= \left[ \begin{array}{cccc}
{I}_{max} - {I}_{k-j+2}(t_{k-j}) \\
{I}_{max} - {I}_{k-j+3}(t_{k-j}) \\
\vdots \\
{I}_{max} - {I}_{d}(t_{k-j}) \\
\end{array} \right]
\end{displaymath}
\small}
\label{table:matrix_form}
\caption{Second set of constraints in (\ref{eq:stagek-j-2}) in Matrix form.} 
\vspace{-0.2in}
\end{table*}
\end{center}

Now consider the $(k-j)^{th}$ stage and suppose Theorem \ref{thm:min_delay} holds for all stages after stage $(k-j)$ where $2 \leq j \leq k-1$. 
This means that in every stage after stage $(k-j)$, only the node that has just decoded the packet transmits.
At time $t_{k-j}$, all nodes except $k-j+1, k-j+2, \ldots, k$ and $d$ have decoded the packet. 
Let the mutual information
state at these nodes at time $t_{k-j}$ be ${I}_{k-j+1}(t_{k-j}), {I}_{k-j+2}(t_{k-j}), \ldots, {I}_{k}(t_{k-j})$ and 
${I}_{d}(t_{k-j})$, respectively. 
Also, suppose in the $(k-j)^{th}$ stage,
the source and the relay nodes $1, 2, \ldots, k-j$ transmit a fraction $\alpha_0^{k-j}, \alpha_1^{k-j}, \alpha_2^{k-j}, \ldots, \alpha_{k-j}^{k-j}$
of the total duration of stage $(k-j)$, i.e., $\Delta_{k-j}$ respectively. 
Then, the optimal solution must solve the following optimization problem:
\begin{align}
&\textrm{Minimize:} \; \sum_{m=0}^j \Delta_{k-j+m} \nonumber \\
& \textrm{Subject to:} \; {I}_{k-j+1}(t_{k-j}) + \Delta_{k-j}\Big[ \sum_{i=0}^{k-j} \alpha_i^{k-j} C_{i,k-j+1} \Big] = {I}_{max} \nonumber \\
& {I}_{k-j+n}(t_{k-j}) + \Delta_{k-j}\Big[ \sum_{i=0}^{k-j} \alpha_i^{k-j} C_{i,k-j+n} \Big] \nonumber\\
& + \sum_{i=1}^{n-1} \Delta_{k-j+i} C_{k-j+i, k-j+n}  = {I}_{max} \; \forall n \in \{2, \ldots, j+1\} \nonumber \\
& 0 \leq \alpha_0^{k-j} , \alpha_1^{k-j}, \ldots, \alpha_{k-j}^{k-j} \leq 1,  \sum_{i=0}^{k-j} \alpha_i^{k-j} = 1 \nonumber \\
& \Delta_{k-j} \geq 0, \Delta_{k-j+1} \geq 0, \ldots, \Delta_k \geq 0
\label{eq:stagek-j-1}
\end{align}
where the first constraint states that relay $k-j+1$ must accumulate ${I}_{max}$ bits of mutual information
by the end of stage $(k-j)$.
The second set of constraints state that every subsequent node $k-j+n$ (where $2 \leq n \leq j+1$) 
including the destination in the ordering $\mathcal{O}_{opt}$
must accumulate ${I}_{max}$ bits of mutual information
by the end of stage $(k-j+n)$. In the last term of the left hand side of each such constraint, 
we have used the induction hypothesis that in every stage after stage $(k-j)$, only the node that just decoded the packet transmits. 
Using the transform $\beta_i = \Delta_{k-j}\alpha_i^{k-j}$ for all $i \in \{0, 1, 2, \ldots, k-j\}$, and 
$\sum_{i=0}^{k-j} \alpha_i^{k-j} = 1$, we have the equivalent problem:
\begin{align}
&\textrm{Minimize:} \; \sum_{i=0}^{k-j} \beta_i + \Delta_{k-j+1} + \ldots + \Delta_{k-1} + \Delta_k \nonumber \\
&\textrm{Subject to:}  \; {I}_{k-j+1}(t_{k-j}) + \sum_{i=0}^{k-j} \beta_i C_{i,k-j+1} = {I}_{max} \nonumber \\
& {I}_{k-j+n}(t_{k-j}) + \sum_{i=0}^{k-j} \beta_i C_{i,k-j+n} \nonumber\\
& + \sum_{i=1}^{n-1} \Delta_{k-j+i} C_{k-j+i, k-j+n} = {I}_{max} \; \forall n \in \{2, \ldots, j+1\} \nonumber \\
& \beta_i \geq 0 \; \forall i \in \{0, 1, 2, \ldots, k-j\}, \Delta_{k-j+1} \geq 0, \ldots, \Delta_k \geq 0 
\label{eq:stagek-j-2}
\end{align}
The problems (\ref{eq:stagek-j-1}) and (\ref{eq:stagek-j-2})
are equivalent because we can transform (\ref{eq:stagek-j-2}) to the original problem by using the relations
$\Delta_{k-j} = \sum_{i=0}^{k-j} \beta_i$ and $\alpha_i^{k-j} = \frac{\beta_i}{\Delta_{k-j}}$. The
degenerate case where $\Delta_{k-j} = 0$ does not arise because if $\Delta_{k-j} = 0$, then no node transmits in stage $(k-j)$. We 
know from the induction hypothesis that only the nodes after node $k-j$ in the ordering $\mathcal{O}_{opt}$ transmit after stage $(k-j)$.
This means that node $k-j$ never transmits, a contradiction.


The second set of constraints in problem (\ref{eq:stagek-j-2}) can be written in matrix form 
as $\mathbf{B} + \mathbf{C\Delta} = \mathbf{I}$ as shown in Table I
where  $\mathbf{C}$ is a lower triangular matrix with
all diagonal entries being non-zero.
{Note that if we are using the optimal relay set such that each node 
transmits for a non-zero duration, then none of the capacity values in the diagonal of $\mathbf{C}$ can be zero. 
This is because each such term denotes the capacity value between successive relay nodes in the optimal transmission schedule. 
If any of these is zero, it would imply that a successor relay node is able to decode the packet without requiring its predecessor node to transmit. 
Then, one can simply remove the predecessor node from the optimal relay set, leading to a contradiction.}

Thus, we have:
$\mathbf{\Delta} = \mathbf{C}^{-1} (\mathbf{I} - \mathbf{B})$. Therefore 
each of the terms $\Delta_{k-j+1}, \Delta_{k-j+2}, \ldots,  \Delta_{k-1}, \Delta_k$
is linear in the variables $\beta_0, \beta_1, \ldots, \beta_{k-j}$. 
Using this, the objective in (\ref{eq:stagek-j-2}) can be expressed as a 
linear function of these variables. Let this be denoted by $f(\beta_0, \beta_1, \ldots, \beta_{k-j})$. 
Also we know that under the optimal solution, 
$\Delta_{k-j+1} > 0, \ldots, \Delta_k > 0$. Thus, 
we can remove the last set of constraints from (\ref{eq:stagek-j-2}) without affecting the optimal solution (using Lemma \ref{lem:inactive_const}).
Thus, (\ref{eq:stagek-j-2}) becomes:
\begin{align}
\textrm{Minimize:} \qquad & f(\beta_0, \beta_1, \ldots, \beta_{k-j}) \nonumber \\
\textrm{Subject to:}  \qquad & {I}_{k-j+1}(t_{k-j}) + \sum_{i=0}^{k-j} \beta_i C_{i,k-j+1} = {I}_{max} \nonumber \\
& \beta_i \geq 0 \qquad \forall i \in \{0, 1, 2, \ldots, k-j\}
\label{eq:stagek-j-3}
\end{align}
Similar to the stage $(k-1)$ case, 
this optimization problem is linear in $\beta_i$ with a single linear equality constraint and thus
the solution is of the form where all except one $\beta_i$ are zero.
Since $\alpha_i^{k-j} = \frac{\beta_i}{\Delta_{k-j}}$, we have that in the optimal solution, exactly one of the fractions 
$\alpha_0^{k-j} , \alpha_1^{k-j}, \ldots, \alpha_{k-j}^{k-j}$
is equal to $1$ and rest must be $0$. This implies that only one node transmits in this stage. Further, 
this node must be the relay node $k-j$ that decoded the packet at the beginning of this stage. Else,
node $k-j$ never transmits. This is because by definition of stage $(k-j)$, node $k-j$ does not have the
packet before the
beginning of stage $(k-j)$ and hence cannot transmit before stage $(k-j)$.
By induction hypothesis, only nodes $k-j+1, k-j+2, \ldots, k$ transmit when stage $(k-j)$
ends. Thus, if node $k-j$ is not the node chosen for stage $(k-j)$, it never
transmits, contradicting the fact that it is part of the optimal set. 
This proves the Theorem.
\end{proof}

\subsection{Exact Solution for a Line Topology}
\label{section:line_network}

In this section, we present the optimal solution for a special case of
line topologies. Specifically, all nodes are located on a line as shown in Fig. \ref{fig:four} and
no two nodes are co-located.
We assume that each node transmits at the same PSD $P$. Further, the transmission capacity $C_{ij}$ between any two nodes $i$ and $j$ 
depends only on the distance $d_{ij}$ between the two nodes and is a monotonically decreasing function of
$d_{ij}$. For example, we may have that $C_{ij} = \log_2(1 + \frac{h_{ij}P}{N_0})$ where $P$ is the PSD and
$h_{ij} = \frac{1}{d_{ij}^\alpha}$ where $\alpha \geq 2$ is the path loss coefficient. Under these assumptions, we can determine the
optimal cooperating set for the problem of routing with mutual information accumulation as follows.

\begin{lem}
The optimal cooperating set for the line topology as described above is given by the set of all relay nodes located
between the source and the destination.
\label{lem:line_nw}
\end{lem}


\begin{proof} \emph{(Lemma \ref{lem:line_nw}):}
Consider the line network as shown in Fig. \ref{fig:four}. 
We first show that the optimal cooperating set cannot contain any relay node that lies to the left of the source.
Suppose the optimal set contains one or more such nodes. Then, we can replace 
all transmissions by these nodes with a source transmission and
get a smaller delay. This is because the source has a strictly higher transmission capacity to \emph{all} nodes to its right than
each of these nodes.

Next, we show that the optimal cooperating set must contain all the nodes that are located between $s$ and $d$.
We know that $s$ is the first node to transmit. The first relay node that decodes the packet is node $1$, since 
link $s-1$ has the smallest distance and therefore the highest transmission capacity among all links from
$s$ to nodes to the right of $s$. From Theorem \ref{thm:min_delay}, we know that once node $1$ has decoded the packet,
it should start transmitting if it is part of the optimal set. Else, it never transmits and the source continues to transmit
until another node can decode the packet. Suppose that the optimal set does not contain node $1$. Then, we can get a smaller delay
by having node $1$ transmit instead of $s$ once it has decoded the packet. This is because node $1$ has
a strictly higher transmission capacity to \emph{all} nodes to its right than $s$. Thus, we have that the optimal set
must contain node $1$.

The above argument can now be applied to each of the nodes $2, 3, \ldots, n$ as in Fig. \ref{fig:four}. 
This proves the Lemma. 
\end{proof}


\begin{figure}
\centering
\includegraphics[width=7cm, angle=0]{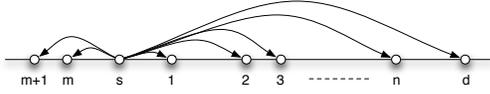}
\caption{A line topology.}
\label{fig:four}
\vspace{-0.1in}
\end{figure}

To get an idea of the reduction in delay achieved by using mutual information accumulation over traditional routing, consider
the line topology  above with $n$ nodes placed between $s$ and $d$ at equal distance such that $d_{i, i+1} = 1$ for all $i$. 
Also, suppose the transmission capacity on link $i-j$ is given by $C_{ij} = \frac{\gamma P}{d_{ij}^2}$ where $\gamma > 0$ is a constant. 
Then the capacity of link $s-1$ is ${\gamma P}$, the capacity of link $s-2$ is $\frac{\gamma P}{4}$, 
the capacity of link $s-3$ is $\frac{\gamma P}{9}$, and so on. Define $\theta \defequiv {\gamma P}$.
Then, the minimum delay for routing with mutual information accumulation is given by
$\sum_{i=0}^n \Delta_i$ where:
\begin{align*}
\Delta_0 &= \frac{I_{max}}{C_{s1}} = \frac{I_{max}}{\theta},
\Delta_1 = \frac{I_{max} - \Delta_0 C_{s2}}{C_{12}} = \frac{I_{max} - \Delta_0 \frac{\theta}{4}}{\theta} \\
& \vdots \\
\Delta_n &= \frac{I_{max} - \sum_{i=0}^{n-1}\Delta_i C_{i,n+1}}{C_{n, n+1}} = \frac{I_{max} - \sum_{i=0}^{n-1}\Delta_i \frac{\theta}{(n+1-i)^2}}{\theta}\\
\end{align*}
For simplicity, let us ignore the contribution of nodes that are more than $3$ units away from a receiver. Then, we have:
\begin{align*}
&\sum_{i=0}^n \Delta_i = \frac{(n+1) I_{max} - \frac{\theta}{4}\sum_{i=0}^{n-1} \Delta_i - \frac{\theta}{9}\sum_{i=0}^{n-2} \Delta_i}{\theta} \\
&\Rightarrow \sum_{i=0}^{n} \Delta_i = \frac{(n+1) I_{max} + \frac{\theta}{4} \Delta_n + 
\frac{\theta}{9}(\Delta_n + \Delta_{n-1})}{\theta(1 + \frac{1}{4} + \frac{1}{9})} \\
&< \frac{(n+1) I_{max} + \frac{\theta}{4} \Delta_0 + \frac{\theta}{9} 2 \Delta_0}{\theta(1 + \frac{1}{4} + \frac{1}{9})} 
= \frac{I_{max}}{\theta} \Bigg(\frac{n + 1 + \frac{1}{4} + \frac{2}{9}}{1 + \frac{1}{4} + \frac{1}{9}} \Bigg)
\end{align*}
where we used the fact that $\Delta_n, \Delta_{n-1} < \Delta_0$.
The minimum delay for traditional routing is simply $(n+1) \Delta_0 = (n+1)\frac{I_{max}}{\theta}$. Thus, for this network, the delay under
mutual information accumulation is smaller than that under traditional routing  
at least by a factor $\frac{n + 1 + \frac{1}{4} + \frac{2}{9}}{(n+1)(1 + \frac{1}{4} + \frac{1}{9})}$ that approaches $\frac{36}{49} = 73\%$ for large $n$.

\section{Minimum Energy Routing with Delay Constraint}
\label{section:related_problem}

Next, we consider the problem of minimizing the sum total energy to transmit a packet from the source to
destination using mutual information accumulation subject to a given delay constraint $D_{max}$. 
This problem is more challenging than problem (\ref{eq:min_delay}) since in addition
to optimizing over the cooperating relay set and the transmission order of nodes in that set, 
it also involves determining the PSD values to be used for
each node. Further, a cooperating relay node may need to transmit at different PSD levels during different stages of the transmission schedule.

\subsection{Problem Formulation}
\label{section:formulation2}

Consider a transmission strategy (similar to the one discussed in Section \ref{section:timeslot}) that is described by a 
cooperating relay set $\mathcal{R}$ of size $|\mathcal{R}| = k$ and
a decoding order $\mathcal{O}$. Let the terms $\Delta_j$ and $A_{ij}$ be defined in a similar fashion. Also, let
$P_{ij}$ denote the PSD at which node $i$ transmits in stage $j$. Then for any transmission strategy $\mathcal{G}$ 
that uses the subset of relay nodes $\mathcal{R}$ with an ordering $\mathcal{O}$, 
the minimum sum total energy to transmit a packet from source to destination subject to the
delay constraint $D_{max}$ is given by the solution to the following optimization problem:
\begin{align}
\textrm{Minimize:} & \sum_{j = 0}^k \sum_{i=0}^j A_{ij} P_{ij} \nonumber \\
\textrm{Subject to:}  & \sum_{j = 0}^k \Delta_j \leq D_{max} \nonumber \\
& \sum_{i=0}^{m-1} \sum_{j=0}^{m-1} A_{ij} C_{im}(P_{ij}) \geq {I}_{max} \; \forall m \in \{1, \ldots, k+1\} \nonumber \\
& \sum_{i=0}^{j} A_{ij} \leq \Delta_j \; \forall j \in \{0, 1, 2, \ldots, k\} \nonumber \\
& A_{ij}, P_{ij} \geq 0 \; \forall i \in \{0, 1, 2, \ldots, k\}, j \in \{0, 1, \ldots, k\} \nonumber \\
& A_{ij} = 0, P_{ij} = 0 \; \forall i > j \nonumber \\
& \Delta_{j} \geq 0 \; \forall j \in \{0, 1, 2, \ldots, k\}
\label{eq:min_total_energy}
\end{align}
where the first constraint represents requirement that the total delay must not exceed $D_{max}$. 
The second constraint captures the requirement that node $m$ in the ordering  must accumulate at least ${I}_{max}$ amount of
mutual information by the end of stage $m-1$ using all transmissions in all stages up to stage $m-1$.
Note that in the second constraint,
$C_{im}(P_{ij})$ denotes the transmission capacity of link $i-m$ in stage $j$ and it is a function of $P_{ij}$, 
the PSD of node $i$ in stage $j$.  Also note that (\ref{eq:min_total_energy}) is not a linear program in general, since
the $C_{im}(P_{ij})$ may be non-linear in $P_{ij}$.

\subsection{Characterizing the Optimal Solution of (\ref{eq:min_total_energy})}
\label{section:central2}

Let $\mathcal{R}_{opt}$ denote the subset of relay nodes that take part in the routing process in the optimal solution. Let 
$k = |\mathcal{R}_{opt}|$ be the size of this set. Also, let $\mathcal{O}_{opt}$ be the optimal ordering.
Note that, by definition, each node in $\mathcal{R}_{opt}$ transmits for a non-zero duration (else, we can remove it from the set without
affecting the sum total energy). Finally, let $P_{ij}^{opt}$ denote the optimal PSD used by node $i$ in stage $j$.
Then, similar to Theorem \ref{thm:min_delay}, we have the following:


\begin{thm}
Under the optimal solution to the minimum energy routing with delay constraint problem $(\ref{eq:min_total_energy})$, 
in each stage $j$, it is optimal for only one node to transmit, and that node is node $j$. 
\label{thm:optimal_sol_min_energy}
\end{thm}

\begin{proof} \emph{(Theorem \ref{thm:optimal_sol_min_energy}):}
Note that given any fixed set of power levels $P_{ij}$, the problem (\ref{eq:min_total_energy}) becomes linear in the other variables and
is similar to the problem (\ref{eq:min_delay}). Thus, similar arguments as in the proof of Theorem \ref{thm:min_delay} can be applied to establish this
property and we omit the details for brevity.
\end{proof}

Although Theorem \ref{thm:optimal_sol_min_energy} simplifies the optimization problem (\ref{eq:min_total_energy}), 
it does not yield a greedy transmission strategy applied over all subsets (similar to the one in Section \ref{section:central}) 
for computing the optimal solution.  
This is because the transmission order generated by the greedy strategy depends on the power levels used. 
For general non-linear rate-power functions, different power levels can give rise to different decoding orders for the same relay set
under the greedy strategy (see Appendix B for an example).
Thus, solving (\ref{eq:min_total_energy}) may involve searching over all possible orderings of all possible subsets.
However, for the special, yet important case of \emph{linear} rate-power functions, this problem can be simplified considerably. 
A linear rate-power function is a good approximation for the low SNR regime. For example, in sensor networks where bandwidth is 
plentiful and power levels are small, it is reasonable to assume that the nodes operate in the low SNR regime.
In the following, we will assume that the transmission capacity $C_{ij}(P_i)$ on link $i-j$ is given by $C_{ij}(P_{i}) = \gamma P_i h_{ij}$
(in units of bits/sec/Hz) where $\gamma$ is a constant and $P_i$ is the PSD of node $i$. Then, we have the following:

\begin{thm}
For linear rate-power functions, 
the decoding order of nodes in the optimal set $\mathcal{R}_{opt}$ under the greedy transmission strategy
is the same for all non-zero power allocations. Further, the sum total power required to
transmit a packet from the source to the destination is the same for all non-zero power allocations.
\label{thm:optimal_sol_linear}
\end{thm}

\begin{proof} \emph{(Theorem \ref{thm:optimal_sol_linear}):}
We prove by induction. Consider any non-zero power allocation used by the nodes in $\mathcal{R}_{opt}$. 
The source is the first node to transmit. Let it be indexed by $0$.
Also, suppose the source uses PSD $P_0 > 0$. 
Under the greedy transmission strategy, the
source continues to transmit until any node can decode the packet.
This node is the one that minimizes 
$\Delta_0 = \frac{I_{max}}{C_{0i}(P_0)} = \frac{I_{max}}{\gamma P_0 h_{0i}}$ over all $i \in \mathcal{R}_{opt}$, 
which is the time to decode the packet. 
Clearly, this node is the same for all $P_0 > 0$. Let it be indexed by $1$. Also, we have that:
\begin{align*}
\Delta_0 = \frac{I_{max}}{\gamma P_0 h_{01}} \Rightarrow \Delta_0 P_0 = \frac{I_{max}}{\gamma h_{01}}
\end{align*}
which shows that the total power used in stage $0$ is independent of $P_0$.
Next, let the PSD of node $1$ be $P_1$. Then, in stage $1$ under the greedy transmission strategy, 
node $1$ transmits until any node that does not have the packet yet can decode it.
This node is the one that minimizes over all $i \in \mathcal{R}_{opt} \setminus \{1\}$:
\begin{align*}
\frac{I_{max} - \Delta_0 C_{0i}(P_0)}{C_{1i}} = 
\frac{I_{max} - \Delta_0 \gamma P_0 h_{0i}}{\gamma P_1 h_{1i}} = 
\frac{I_{max}(1 - \frac{h_{0i}}{h_{01}})}{\gamma P_1 h_{1i}} 
\end{align*}
Clearly, this node is the same for all $P_1 > 0$. Let it be indexed by $2$. Also, we have that:
\begin{align*}
\Delta_1 
= \frac{I_{max}(1 - \frac{h_{02}}{h_{01}})}{\gamma P_1 h_{12}} 
\Rightarrow \Delta_1 P_1 = \frac{I_{max}(1 - \frac{h_{02}}{\gamma h_{01}})}{\gamma h_{12}}
\end{align*}
which shows that the total power used in stage $1$ is independent of $P_0$ and $P_1$.

Now suppose this holds for all stages $\{0, 1, 2, \ldots, j-1\}$ where $j-1 < k$. We show that it also holds for stage $j$.
Let the PSD of node $j$ be $P_j$.
Under the greedy strategy, node $j$ continues to transmit in stage $j$ until any node that does not have the packet yet can decode it.
This node is the one that minimizes over all $i \in \mathcal{R}_{opt} \setminus \{1, 2, \ldots, j\}$:
\begin{align*}
\frac{I_{max} - \sum_{m=0}^{j-1} \Delta_m C_{mi}(P_m)}{C_{ji}(P_j)} 
= \frac{I_{max} - \gamma \sum_{m=0}^{j-1} \Delta_m P_m h_{mi}}{\gamma P_j h_{ji}} 
\end{align*}
From the induction hypothesis, we know that each of the terms $\Delta_m P_m$ for all $m \in \{0, 1, \ldots, j-1\}$
is independent of the power levels $P_m$. Thus, we have that the node that minimizes the expression above is the same
for all $P_j > 0$. Further, the total power used in stage $j$ is given by
\begin{align*}
\Delta_j P_j = \frac{I_{max} - \gamma \sum_{m=0}^{j-1} \Delta_m P_m h_{mi}}{\gamma h_{ji}} 
\end{align*}
which is independent of $P_0, P_1, \ldots, P_j$. This proves the Theorem.
\end{proof}

\subsection{A Greedy Algorithm}
\label{section:greedy2}

Theorem \ref{thm:optimal_sol_min_energy} suggests a simple method for computing the optimal solution to (\ref{eq:min_total_energy}) 
when the rate-power function is linear.
Specifically, we start by setting all PSD levels to the same value, say some $P > 0$. From Theorem \ref{thm:optimal_sol_linear}, we know that
the sum total power required to transmit a packet from the source to the destination is the same for all non-zero power allocations.
Then, solving (\ref{eq:min_total_energy}) is equivalent to solving the minimum delay problem (\ref{eq:min_delay}) 
with given power levels, except the delay constraint. 
This can be done using the greedy strategy described in Section \ref{section:central}.
If the solution obtained satisfies the delay constraint $D_{max}$, then we are done. 
Else, suppose we get a delay $D > D_{max}$. Then, we can scale up the power level $P$ by a factor $\frac{D}{D_{max}}$ and scale down
the duration of each stage $\Delta_j$ by the same factor. This ensures that the delay constraint is met while the sum total power used
remains the same.


\section{Minimum Delay Broadcast}
\label{section:third_problem}

Next, we consider the problem of minimum delay broadcast for the network model described in Section \ref{section:model}.
In this problem, starting with the source node, the goal is to deliver the packet to all nodes in the network in minimum time
with mutual information accumulation. We assume that there are $n$ nodes in the network other than the source. 
Similar problems have been considered in \cite{hitch-hiking, maric_broadcast, scaglione} which focus on energy accumulation and
where the goal is to broadcast the packet to all nodes using minimum sum total energy.

\subsection{Timeslot and Transmission Structure}
\label{section:timeslot_bc}

For the minimum delay broadcast problem, the transmission strategy and resulting time timeslot structure under a general policy 
is similar to the one discussed for the minimum delay routing problem in Section \ref{section:timeslot}.
Specifically, let $\mathcal{O}$ be the ordering of the $n$ nodes that represents the sequence in which 
they successfully decode the packet under any strategy.
Without loss of generality, let the nodes in the ordering $\mathcal{O}$ be indexed
as $1, 2, 3, \ldots, n$. 
Also, let the source $s$ be indexed as $0$. 
Initially, only the source has the packet.
Let $t_0$ be the time when it starts its transmission and let $t_1, t_2, \ldots, t_n$ denote the times 
when nodes $1, 2, \ldots, n$ in the ordering $\mathcal{O}$ accumulate enough mutual information to decode the packet.
We say that the transmission occurs over $n$ stages,
where stage $j, j \in \{ 0, 1, 2, \ldots, n-1\}$ represents the interval $[t_j, t_{j+1}]$.
 Note that in any stage $j$, the first $j$ nodes in the ordering 
$\mathcal{O}$ and the source have the fully decoded packet. Thus, any subset of these nodes (including potentially all of them) 
may transmit during this stage. 
For each $j$, define the duration of stage $j$ as $\Delta_j = t_{j+1} - t_j$. Also, let $A_{ij}$ denote the transmission 
duration for node $i$ in stage $j$. 
As before, we have that $A_{ij} = 0$ if $i > j$, else $A_{ij} \geq 0$. 
The total time to deliver the packet to all the $n$ nodes is given by  $T_{tot} = t_{n} - t_0 = \sum_{j = 0}^{n-1} \Delta_j$.

\subsection{Problem Formulation}
\label{section:formulation_bc}

For any transmission strategy that results in the decoding order $\mathcal{O}$, the minimum delay for broadcast is 
given by the solution to the following optimization problem:
\begin{align}	
&\textrm{Minimize:} \;   T_{tot} = \sum_{j = 0}^{n-1} \Delta_j \nonumber \\
&\textrm{Subject to:} \nonumber\\
& \sum_{i=0}^{m-1} \sum_{j=0}^{m-1} A_{ij} C_{im} \geq {I}_{max} \; \forall m \in \{1, 2, \ldots, n\} \nonumber \\
& \sum_{i=0}^{j} A_{ij} \leq \Delta_j \; \forall j \in \{0, 1, 2, \ldots, n-1\} \nonumber \\
& A_{ij} \geq 0 \; \forall i \in \{0, 1, 2, \ldots, n-1\}, j \in \{0, 1, 2, \ldots, n-1\} \nonumber \\
& A_{ij} = 0 \; \forall i > j \nonumber \\
& \Delta_{j} \geq 0 \; \forall j \in \{0, 1, 2, \ldots, n-1\}
\label{eq:min_delay_broadcast}
\end{align}
This is similar to (\ref{eq:min_delay}) except that the set $\mathcal{R}$ contains all $n$ nodes and that $d$ is not necessarily the last node to
decode the packet. As in (\ref{eq:min_delay}),
the first constraint captures the requirement that node $m$ in the decoding order $\mathcal{O}$ must accumulate at least ${I}_{max}$ amount of mutual information
by the end of stage $m-1$ using transmissions in all stages up to stage $m-1$. 
The second constraint means that in every stage $j$, the total transmission time for all nodes that have the
fully decoded packet in that stage cannot exceed the length of that stage.
Similar to (\ref{eq:min_delay}), the above problem is a linear program and thus can be solved efficiently for a given ordering $\mathcal{O}$. 
This is the approach taken in \cite{maric_broadcast} (with energy accumulation instead of mutual information accumulation, 
and with the objective of minimizing total energy for broadcast instead of delay) that proposes solving such a linear program for
\emph{every possible ordering} of the $n$ nodes, resulting in $n!$ linear programs.
In the next section, we show that the above computation can be simplified 
by making use of a structural property of the optimal solution that is similar to the results of Theorems 
\ref{thm:min_delay} and \ref{thm:optimal_sol_min_energy}. This results in a greedy algorithm that does not require
solving such linear programs to compute the optimal solution.


\subsection{Characterizing the Optimal Solution of (\ref{eq:min_delay_broadcast})}
\label{section:central_bc}

Let $\mathcal{O}_{opt}$ be the decoding order under the optimal solution. Suppose the 
the nodes in the ordering are labeled as $\{0, 1, 2, \ldots, n-1, n\}$ with $0$ being the source node.
Then, similar to Theorems \ref{thm:min_delay} and \ref{thm:optimal_sol_min_energy}, we have the following:


\begin{thm}
Under the optimal solution to the minimum delay broadcast problem $(\ref{eq:min_delay_broadcast})$,
in each stage $j$, it is optimal for at most one node to transmit. 
\label{thm:min_delay_broadcast}
\end{thm}

While Theorem \ref{thm:min_delay_broadcast} states that under the optimal solution, 
at most one node transmits in each stage $j$, unlike Theorems \ref{thm:min_delay} and \ref{thm:optimal_sol_min_energy}, 
it does not say that this node must be node $j$. 
In fact, this node could be any one of the nodes that have the full packet. 
Specifically, let $r_{j}$ be the node that transmits in stage $j$. Then, using
Theorem \ref{thm:min_delay_broadcast}, we have that $r_j \in \{0, 1, 2, \ldots, j\}$.
The optimal timeslot structure for the minimum delay broadcast problem is shown in
Fig. \ref{fig:broadcast}. Note that unlike the optimal 
timeslot structure for the minimum delay routing problem
(Fig. \ref{fig:three}), here it is possible for
a node to transmit more than once over the course of the broadcast.
This property does not reduce the complexity of finding the optimal
solution from $O(n!)$ linear programs to $O(2^n)$ runs of a greedy strategy. However, as we show in Section \ref{section:greedy2_bc}, 
it still leads to a {greedy} algorithm for finding the
optimal solution that does not require solving $n!$ linear programs like in \cite{maric_broadcast}.

\begin{figure}
\centering
\includegraphics[height=3cm, width=8cm, angle=0]{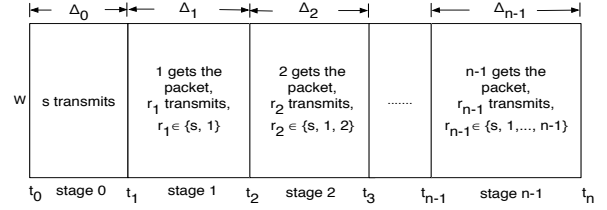}
\caption{Optimal timeslot and transmission structure for minimum delay broadcast. 
In each stage, at most one node from the set of nodes that have the full packet transmits.}
\label{fig:broadcast}
\vspace{-0.2in}
\end{figure}

\begin{proof} \emph{(Theorem \ref{thm:min_delay_broadcast}):}
The proof is similar to the proof of Theorem \ref{thm:min_delay} and therefore, we only provide a sketch here, highlighting
the main differences.

Note that Theorem \ref{thm:min_delay_broadcast} trivially holds in stage $0$ (since only the source has the full packet in this stage).
Next, similar to Theorem \ref{thm:min_delay}, in the last stage (i.e., stage $(n-1)$), only the node with the \emph{best} link (in terms of transmission
capacity) to node $n$ in the ordering $\mathcal{O}_{opt}$ should transmit in order to minimize the total delay. 
Let this node be labeled $r_{n-1}$. However, unlike Theorem \ref{thm:min_delay}, 
we cannot claim that this node must be node $n-1$ in the ordering $\mathcal{O}_{opt}$. This is because
while $r_{n-1}$ has the best link to $n$, it does not necessarily have the best links to all those nodes in the
decoding order $\mathcal{O}_{opt}$ that come after $r_{n-1}$.
Thus $r_{n-1}$ could be any one of $\{0, 1, 2, \ldots, n-1\}$.
This shows that under the optimal solution, in the last stage $(n-1)$, only one node $r_{n-1}$ transmits.
Using induction, we can show that in every prior stage $(n-j)$ where $1 < j < n$, at most one node needs to transmit.

Consider the $(n-2)^{th}$ stage. At time $t_{n-2}$, all nodes except $n-1$ and $n$ have decoded the packet. Let the mutual information
state at nodes $n-1$ and $n$ at time $t_{n-2}$ be ${I}_{n-1}(t_{n-2})$ and ${I}_n(t_{n-2})$ respectively. 
Also, suppose in the $(n-2)^{th}$ stage,
relay nodes $1, 2, \ldots, n-2$ and the source transmit a fraction $\alpha_1^{n-2}, \alpha_2^{n-2}, \ldots, \alpha_{n-2}^{n-2}$ and 
$\alpha_0^{n-2} $ of the total duration of stage $(n-2)$, i.e., $\Delta_{n-2}$, respectively. 
Note that these
fractions must add to $1$ since it is suboptimal to have any idle time (where no one is transmitting).
Then, the optimal solution must solve the following optimization problem:
\begin{align}
&\textrm{Minimize:} \;  \Delta_{n-2} + \Delta_{n-1} \nonumber \\
&\textrm{Subject to:}  \; \nonumber\\
& {I}_{n-1}(t_{n-2}) + \Delta_{n-2} \sum_{i=0}^{n-2} \alpha_i^{n-2} C_{i, n-1} \geq {I}_{max} \nonumber \\
& {I}_n(t_{n-2}) + \Delta_{n-2} \sum_{i=0}^{n-2} \alpha_i^{n-2} C_{in} + \Delta_{n-1} C_{r_{n-1},n} \geq {I}_{max} \nonumber \\
& 0 \leq \alpha_0^{n-2} , \alpha_1^{n-2}, \ldots, \alpha_{n-2}^{n-2} \leq 1 \nonumber \\
& \sum_{i=0}^{n-2} \alpha_i^{n-2} = 1 \nonumber \\
& \Delta_{n-2} \geq 0,  \Delta_{n-1} \geq 0 
\label{eq:stagek-1-1_bc}
\end{align}
Here, the first constraint states that node $n-1$ must accumulate at least ${I}_{max}$ bits of mutual information
by the end of stage $(n-2)$. 
The second constraint states that node $n$ must accumulate at least ${I}_{max}$ bits of mutual information
by the end of stage $(n-1)$.
Note that in the last term of the left hand side of the second constraint, we have used the fact that only node $r_{n-1}$ transmits during stage $(n-1)$.

It is easy to see that under the optimal solution, the first and second constraints must be met with equality.
This simply follows from the definition of the beginning of any stage $j$ as the time when node $j$ has just decoded the packet. 
Next, let $\beta_i = \Delta_{n-2}\alpha_i^{n-2}$ for all $i \in \{0, 1, 2, \ldots, n-2\}$. 
Since $\sum_{i=0}^{n-2} \alpha_i^{n-2} = 1$, we have that $\sum_{i=0}^{n-2} \beta_i = \Delta_{n-2}$ and (\ref{eq:stagek-1-1_bc}) is equivalent to:
\begin{align}
\textrm{Minimize:} \; & \sum_{i=0}^{n-2}\beta_i + \Delta_{n-1} \nonumber \\
\textrm{Subject to:}  \; & {I}_{n-1}(t_{n-2}) +  \sum_{i=0}^{n-2} \beta_i C_{i, n-1} = {I}_{max} \nonumber \\
& {I}_n(t_{n-2}) + \sum_{i=0}^{n-2} \beta_i C_{in}  + \Delta_{n-1} C_{r_{n-1}, n} = {I}_{max} \nonumber \\
& \Delta_{n-1} \geq 0, \beta_i \geq 0 \qquad \forall i \in \{0, 1, 2, \ldots, n-2\}
\label{eq:stagek-1-3_bc}
\end{align}

Note that problems (\ref{eq:stagek-1-1_bc}) and (\ref{eq:stagek-1-3_bc})
are equivalent because we can transform (\ref{eq:stagek-1-3_bc}) to the original problem by using the relations
$\Delta_{n-2} = \sum_{i=0}^{n-2} \beta_i$ and $\alpha_i^{n-2} = \frac{\beta_i}{\Delta_{n-2}}$. 
In the degenerate case where $\Delta_{n-2} = 0$, we have that no node transmits in stage $(n-2)$, so that
Theorem \ref{thm:min_delay_broadcast} holds. Note that this is different from the scenario in the minimum delay routing
problem where, by definition, the length of a stage in the optimal solution is strictly positive. Here, it is possible for
a stage to have zero length. This happens when two or mode nodes can
accumulate enough mutual information and decode the packet simultaneously.

Using similar arguments as in Theorem \ref{thm:min_delay}, it can be shown that
when $\Delta_{n-2} > 0$, then in the optimal solution exactly one of the fractions 
$\alpha_0^{n-2} , \alpha_1^{n-2}, \ldots, \alpha_{n-2}^{n-2}$
is equal to $1$ and rest must be $0$. This implies that only one node transmits in this stage. 
Combining with the case where $\Delta_{n-2} = 0$, we have that at most 
one node transmits in stage $(n-2)$. We label this node as $r_{n-2}$. Note that
$r_{n-2}$ could be any one of $\{0, 1, 2, \ldots, n-2\}$.

Using induction, it can be shown that in every stage $(n-j), 2 < j < n$, 
at most one node labeled $r_{n-j}$ transmits. Further, 
$r_{n-j}$ could be any one of $\{0, 1, 2, \ldots, n-j\}$.
This proves the Theorem.
\end{proof}

\subsection{A Greedy Algorithm}
\label{section:greedy2_bc}

Theorem \ref{thm:min_delay_broadcast} can be used to construct the following simple algorithm for computing the optimal solution to 
(\ref{eq:min_delay_broadcast}). The algorithm runs in $n$ stages. In each stage $j$, $0 \leq j \leq n-1$, 
the algorithm performs $(j+1)!$ \emph{separate} runs.
Each run corresponds to selecting one transmitter from the set of nodes with the full packet at the start of that run
and having that node transmit until
a new node decodes the packet. Thus, the number of nodes with the full packet increases by one at the end of each run.

Let $\mathcal{S}_{ij}$ denote the set of nodes that have the full packet at the end of the $i^{th}$ run of stage $j$.
We have that the size of $\mathcal{S}_{ij}$ is equal to  $\| \mathcal{S}_{ij} \| = j+2$ for all $i$. 
Further, there are $(j+1)!$ distinct such sets at the end of stage $j$ so that
the algorithm performs $(j+2) \times (j+1)! = (j+2)!$ runs in the next stage $(j+1)$.
  
To see this, note that
we start at stage $0$ with only the source having the full packet and perform only one run. 
At the end of this stage, suppose node $1$ has the packet. Then we have $\mathcal{S}_{10} = \{s, 1\}$.
In next stage (i.e., stage $1$), we perform $2! = 2$ separate runs as follows. 
In the first run, $s$ is chosen as the
transmitter for stage $1$ and continues to transmit until another node (say $x$) gets the packet. This yields
$\mathcal{S}_{11} = \{s, 1, x\}$. 
In the second run, $1$ is chosen as the
transmitter for stage $1$ and continues to transmit until another node (say $y$) gets the packet. This yields
$\mathcal{S}_{21} = \{s, 1, y\}$. Thus, at the end of stage $1$, we have $2!=2$ sets, $\mathcal{S}_{11}$ and $\mathcal{S}_{21}$, 
of size $1+2=3$ each. This procedure is repeated in stage $2$ resulting in $3$ runs starting with $\mathcal{S}_{11}$ and
$3$ runs starting with $\mathcal{S}_{21}$. Thus, in stage $2$, the algorithm performs $(1+2)! = 6$ runs and yields
$3!=6$ sets, $\mathcal{S}_{12}, \mathcal{S}_{22}, \ldots, \mathcal{S}_{62}$, each of size $2+2=4$,
at the end of stage $2$. In the same way, it can be shown that in stage $j$, the algorithm starts with
$j!$ sets of size $j+1$ each, performs $(j+1)!$ runs and results in $(j+1)!$ sets, each of size $j+2$.

The algorithm terminates after stage $(n-1)$ where it performs $n!$ runs and when all nodes decode the packet. The optimal solution is obtained by
picking the sequence of transmitting nodes that yields the minimum delay.

It can be seen that the complexity of this algorithm is $O(n!)$
Essentially, this algorithm performs an exhaustive search over all possible feasible decoding orderings. This corresponds to
searching over all possible values of $r_j \in \{s, 1, 2, \ldots, j\}$ in every stage $j$ (See Fig. \ref{fig:broadcast}).
However, unlike \cite{maric_broadcast}, it does not require solving any linear programs.

\section{Distributed Heuristics and Simulations}
\label{section:distributed}

The greedy algorithm presented in Section \ref{section:central} 
to compute the
optimal solution to problem ($\ref{eq:min_delay}$) has
an exponential computational complexity and is centralized.
In this section, we present two simple heuristics that can be implemented in polynomial time and in a distributed fashion.
We compare the performance of these heuristics with the optimal solution on general network topologies.
We also show the performance of the traditional minimum delay route that does not use mutual information accumulation.


\emph{Heuristic 1}: Here, first the traditional minimum delay route is computed using, say, Dijkstra's shortest path algorithm on the
weighted graph (where the weight $w_{ij}$ of link $i-j$ is defined as the time required to deliver a packet 
from $i$ to $j$, i.e., $w_{ij} = \frac{I_{max}}{C_{ij}}$).
Let $\mathcal{M}$ denote the set of relay nodes that form this minimum delay shortest path.  Then
the greedy algorithm as described in Section \ref{section:central} is applied on the set of nodes in $\mathcal{M}$.
Note that we are not searching over all subsets of $\mathcal{M}$.\footnote{It may be possible to get further gains by 
searching over all subsets of $\mathcal{M}$, but the worst case complexity of doing so would again be exponential. 
Our goal here is to develop polynomial time algorithms.} Thus, the complexity of this heuristic is same as that of any shortest path algorithm, 
i.e., $O(|\mathcal{M}|^2)$. 

\emph{Heuristic 2}: Here, we start with $\mathcal{M}$ as the initial cooperative set. Then, while applying the greedy
algorithm of Section \ref{section:central}, if other nodes that are not in $\mathcal{M}$
happen to decode the packet before the next node
(where the next node is defined as that node in $\mathcal{M}$ that would decode the packet if the current
transmitter continued its transmission), then these nodes are added to the cooperative set if
they have a better channel to the next node than the current transmitter. 
The intuition behind this heuristic is that while $\mathcal{M}$ is expected to be a good cooperative set,
this allows the algorithm to explore more nodes and potentially improve over Heuristic $1$. 




In our simulations, we consider a network of a source, destination, and $n$ relay nodes located in a $10 \times 10$ area.
The location of source $(1.0, 2.0)$ and destination $(8.0, 8.0)$ is fixed while the locations of the other nodes are chosen uniformly at
random. The link gain between any two nodes $i$ and $j$ is chosen from a Rayleigh distribution with mean $1$.
For simplicity, all nodes have the same PSD. The total bandwidth $W$ and packet size $I_{max}$ are both normalized to one unit.
The transmission capacity of link
$i-j$ is assumed to be $C_{ij} = \log_2\Big(1 + \frac{h_{ij}}{d_{ij}^\alpha}\Big)$ where $d_{ij}$ is the distance between nodes $i$ and $j$
and $\alpha$ is the path loss exponent. We choose $\alpha=3$ for all simulations.   

\begin{figure}
\centering
\includegraphics[height=5cm, width=8cm, angle=0]{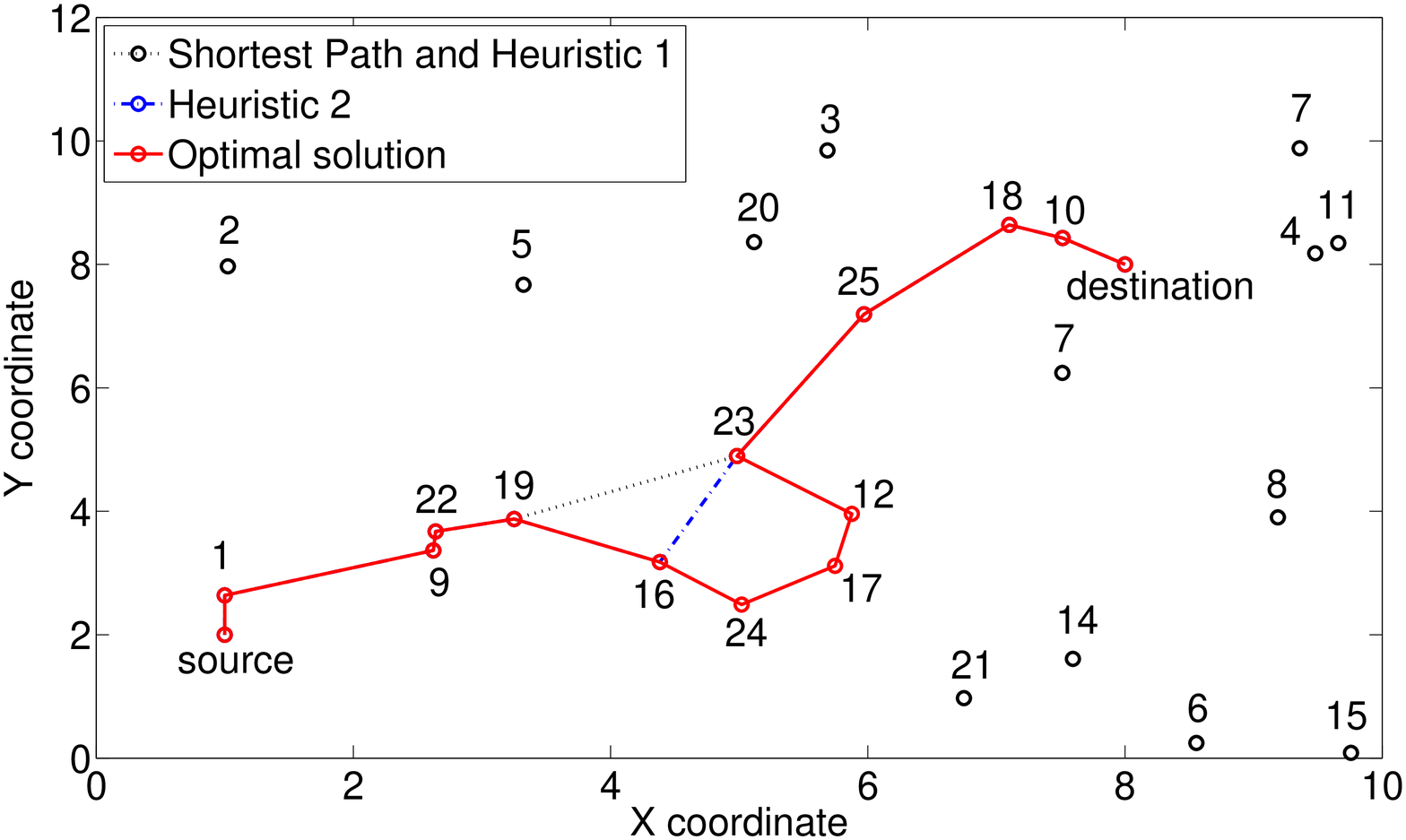}
\caption{A $25$ node network where the routes for traditional minimum delay, Heuristics $1$ and $2$, and 
optimal mutual information accumulation are shown.}
\label{fig:five}
\vspace{-0.1in}
\end{figure}

In the first simulation, $n=25$ and the network topology is fixed as shown in Fig. \ref{fig:five}. We then compute the
traditional minimum delay route and the optimal solution for routing with mutual information accumulation using the
greedy algorithm of Section \ref{section:central}. 
We also implement Heuristics $1$ and $2$ on this network. Fig. \ref{fig:five} shows the results. It is seen that
the traditional minimum delay route is given by $[s, 1, 9, 22, 19, 23, 25, 18, 10, d]$ while the
optimal mutual information accumulation route (according to the decoding order) is given by 
$[s, 1, 9, 22, 19, 16, 24, 17, 12, 23, 25, 18, 10, d]$. The decoding order of nodes under Heuristic $1$ is same as that under
the traditional minimum delay route while that under Heuristic $2$ is given by $[s, 1, 9, 22, 19, 16, 23, 25, 18, 10, d]$.
The total delay under traditional minimum delay routing, Heuristic $1$, Heuristic $2$, and optimal mutual information accumulation routing
was found to be $29.84, 23.73, 22.99$ and $22.19$ seconds respectively. 

This example demonstrates that the optimal route under mutual information accumulation can be quite different from the traditional
minimum delay path. It is also interesting to note that the set of nodes in $\mathcal{M}$ is a subset of the cooperative relay set in this
example. However, this does not hold in general. We also note that the delay under both Heuristics $1$ and $2$ is close to the optimal value.
Finally, while Heuristic $1$ only uses the nodes in $\mathcal{M}$, Heuristic $2$ explores more and ends up using node $16$ as well.


In the second simulation, we choose $n=20$. The source and destination locations are fixed as before but the locations of
the relay nodes are varied randomly over $100$ instances. For each topology instance, we compute the minimum delay obtained by 
these $4$ algorithms. In Fig. \ref{fig:six}, we plot the cumulative distribution function (CDF) of the ratio of the minimum delay 
under the two heuristics and the traditional shortest path to the minimum delay under the optimal mutual information accumulation solution. 
From this, it can be seen that both Heuristic $1$ and $2$ perform quite well over general network topologies. 
In fact, they are able to achieve the optimal performance $40\%$ and
$60\%$ of the time respectively. Further, they are within $10\%$ of the optimal at least $90\%$ of the time and
within $15\%$ of the optimal at least $98\%$ of the time.
Also, Heuristic $2$ is seen to  outperform Heuristic $1$ in general. 
Finally, the average delay in routing with mutual information accumulation 
was found to be $77\%$ of the average delay of traditional shortest path routing (where the average is taken over
the $100$ random topologies). 
It is interesting to note that once the optimal relay set is computed, the only overhead of a mutual information accumulation based scheme over 
traditional shortest path is in the decoding and re-encoding operations. In particular, such schemes only require a 1-bit feedback from a
receiver when it has successfully decoded the packet and do not require sophisticated synchronization, coordination, or extensive channel state information feedback.

\begin{figure}
\centering
\includegraphics[height=4cm, width=8cm, angle=0]{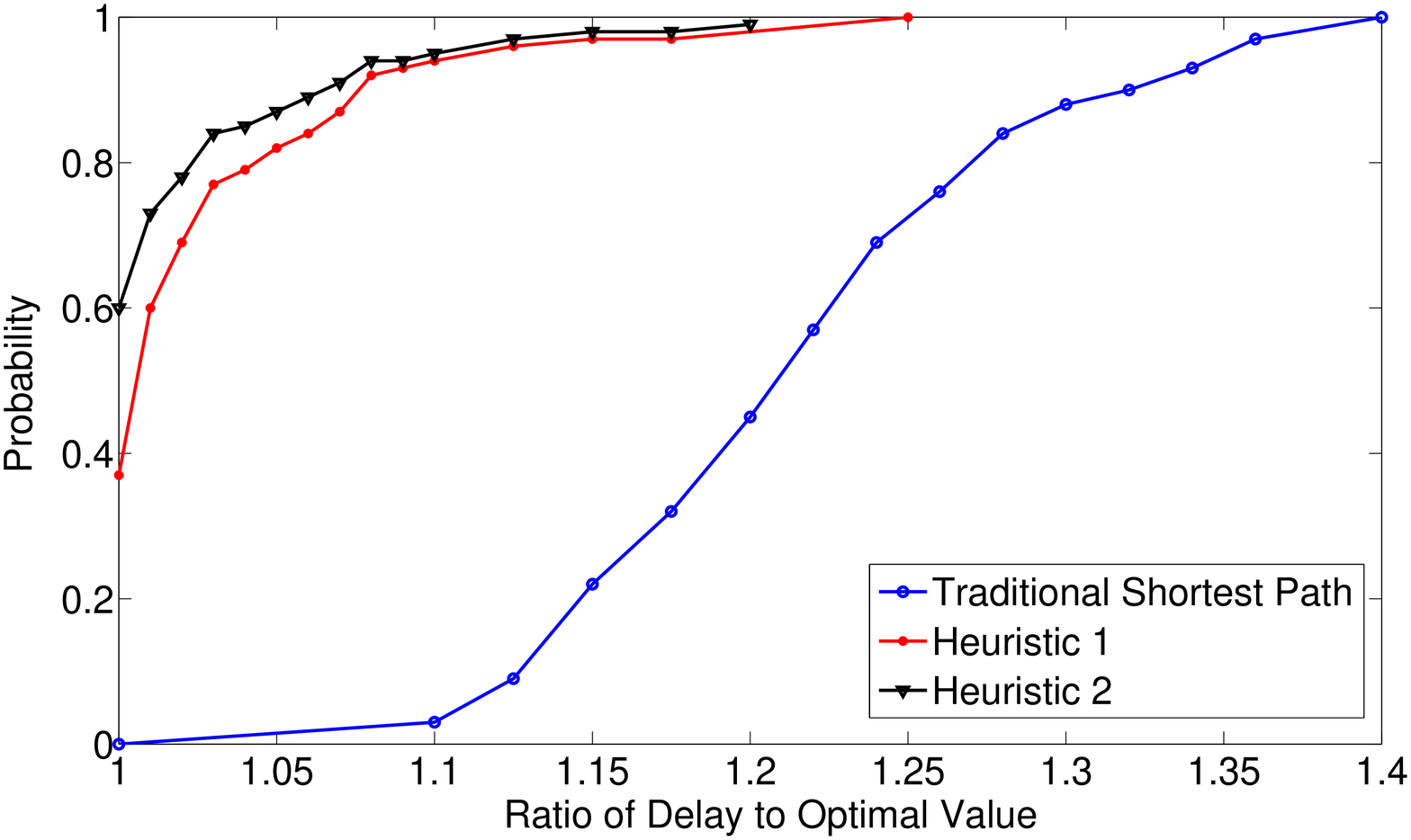}
\caption{The CDF of the ratio of the minimum delay under the two heuristics and the traditional shortest path to the 
minimum delay under the optimal mutual information accumulation solution.}
\label{fig:six}
\vspace{-0.1in}
\end{figure}

\section{Conclusions}

In this work, we considered three problems involving optimal routing and scheduling
over a multi-hop wireless network using mutual information accumulation. We formulated
the general problems as combinatorial optimization problems and then made use of several
structural properties to simplify their solutions and derive optimal greedy algorithms.
A key feature of these algorithms is that unlike prior works on these problems, they
do not require solving any linear programs to compute the optimal solution.
While these greedy algorithms still have exponential complexity,
they are significantly simpler than prior schemes and 
allow us to compute the optimal solution as a benchmark.
We also proposed two simple and practical heuristics that exhibit
very good performance when compared to the optimal solution.

\section*{Appendix A \\ Proof of Lemma \ref{lem:inactive_const}}


We argue by contradiction. 
Suppose an optimal solution to (\ref{eq:inactive}) without the
constraint $x \geq 0$ is given by $x' \neq x^*$. Then, we have that $c^T x' < c^T x^*$. Further, $x'$ satisfies 
all the constraints we did not remove, but must violate at least one of the constraints that we removed. Thus, we have that
$Ax' = b$ and $x' \not > 0$. Now let $x''$ be a convex combination of $x^*$ and $x'$, i.e., $x'' = \theta x^* + (1-\theta)x'$ where $0 < \theta < 1$. 
We have that $c^T x'' = \theta c^T x^* + (1 - \theta) c^T x'$. 
Since $c^T x' < \theta c^T x^* + (1-\theta) c^T x' < c^T x^*$,
we have that $c^T x' < c^T x'' < c^T x^*$.

Since $x^*$ satisfies the strict inequality constraint $x > 0$ in all entries,
there must be a ball about $x^*$ that still satisfies the constraint $x \geq 0$.
Further, the line segment joining $x^*$ and $x'$ 
intersects this ball. Let us choose $\theta$ such that $x''$ is this point of
intersection.  Then $x''$ still satisfies the constraint $x'' \geq 0$.
However, $c^T x'' < c^T x^*$, which contradicts the fact 
that $x^*$ solves (\ref{eq:inactive}) optimally. 


\section*{Appendix B \\ A Simple Example}

\begin{figure}
\centering
\includegraphics[height=3.5cm, angle=0]{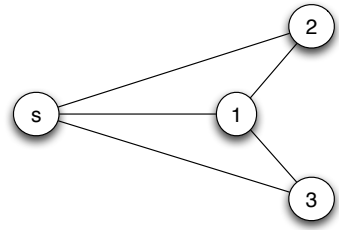}
\caption{The $4$ node example network used in Appendix B.} 
\label{fig:appen}
\end{figure}

Here, we show an example where different power levels can give
rise to different decoding orders for the same relay set under
the greedy transmission strategy when the rate-power curve is non-linear. 
Consider the $4$ node network in Fig. \ref{fig:appen}. We assume the rate-power curves on all links
except link $s-3$ are linear. Specifically, $C_{ij}(P_i) = h_{ij}P_i$ for all $ij \neq s3$. 
However, the rate-power curve on link $s-3$ is logarithmic
and is given by $C_{s3}(P_s) = \log_2(1 + h_{s3}P_s)$. 

Next, suppose $h_{s1} > h_{s2}, h_{s3}$ and $h_{12} = h_{13}$. Also, let
$I_{max} = 1$. Then, node $1$ is the first node to decode the packet for all $P_s > 0$. 
Also, we have $\Delta_0 = \frac{1}{C_{s1}(P_s)} = \frac{1}{{h_{s1}P_s}}$. 

The mutual information state at nodes 
$2$ and $3$ at the end of stage $0$ is given by $I_2(t_1) = \Delta_0 C_{s2}(P_s) = \Delta_0 {h_{s2}P_s}$ and
$I_3(t_1) = \Delta_0 C_{s3}(P_s) = \Delta_0 \log_2(1 + h_{s3}P_s)$ respectively. 
Under the greedy transmission strategy,
after stage $0$, node $1$ will continue to transmit until any of nodes $2$ or $3$ decodes the packet.
Suppose node $1$ uses transmit power $P_1 > 0$. Then,
the time for node $2$ to decode if node $1$ continues to transmit is given by:
\begin{align*}
\delta_2 = \frac{I_{max} - I_2(t_1)}{C_{12}(P_1)} = \frac{1 - \Delta_0 {h_{s2}P_s}}{{h_{12}P_1}} 
= \frac{1 - {\frac{h_{s2}}{h_{s1}}}}{{h_{12}P_1}} 
\end{align*}
Similarly, the time for node $3$ to decode if node $1$ continues to transmit is given by:
\begin{align*}
\delta_3 &= \frac{I_{max} - I_3(t_1)}{C_{13}(P_1)} = \frac{1 - \Delta_0 \log_2(1 + h_{s3}P_s)}{{h_{13}P_1}} \\
&= \frac{1 - \frac{\log_2(1 + h_{s3}P_s)}{h_{s1}P_s}}{{h_{13}P_1}}
\end{align*}
Since $h_{12} = h_{13}$, from the above we have that
$\delta_2 > \delta_3$ if $h_{s2}P_s < \log_2(1 + h_{s3}P_s)$ and 
$\delta_2 < \delta_3$ if $h_{s2}P_s > \log_2(1 + h_{s3}P_s)$. 
Let $h_{s2} = 0.05, h_{s3} = 0.1$. Then, for $P_s = 1$, we get
$\delta_2 > \delta_3$ since $0.05 < \log_2(1.1)$. However, for $P_s = 100$, we have
that $\delta_2 < \delta_3$ since $5 > \log_2(11)$.  This shows that
different power levels can give
rise to different decoding orders for the same relay set under
the greedy transmission strategy when the rate-power curve is non-linear.

\end{document}